\documentclass[a4paper,12pt]{article}
\usepackage[utf8]{inputenc}
\title{On the weighted Gauß--Radau Quadrature}
\author{Sascha Trostorff \& Marcus Waurick\footnote{Institut f\"ur Analysis, TU Dresden, Dresden, Germany, \mbox{e-mail}: sascha.trostorff@tu-dresden.de \&
    Department of Mathematical Sciences, University of Bath,
    Bath, UK.
    \mbox{e-mail}: m.waurick@bath.ac.uk}
}
\date{\today}

\usepackage{fancyhdr} 
\usepackage{amsthm,amsfonts,amsmath,amsopn}
\usepackage{mathtools}
\usepackage{url}

\newcommand{\dd}{\,\mathrm{d}}
\newcommand{\N}{\mathbb{N}}
\newcommand{\R}{\mathbb{R}}

\theoremstyle{plain}
\newtheorem{theorem}{Theorem}[section]
\newtheorem{corollary}[theorem]{Corollary}
\theoremstyle{definition}
\newtheorem*{deff}{Definition}

\begin{document}
  \pagestyle{fancy}
  \maketitle
  \begin{abstract}
    In this short note, we collect some facts on the weighted Gauß-Radau quadrature. In particular, we focus on the location of the Gau\ss--Radau points being a continuous function of the $L^1$-weighting function.
  \end{abstract}
  
  

\section{On the Gau\ss--Radau Quadrature}\label{s:app}
  
  In this technical paper we shall gather some results on the right-sided Gau\ss--Radau 
  quadrature, which are well-known in principle, but are collected for the convenience 
  of the reader, in particular we refer to \cite{FTW16}. For this, we 
  introduce a set of weighting functions:
  \[
    W \coloneqq \{ w\in L^1(-1,1): w>0 \text{ a.e.}\}.
  \]
  Note that the bilinear form
  \[
    \langle \cdot,\cdot\rangle_w \colon D\times D \ni (f,g)\mapsto \int_{-1}^1 f(x)g(x)w(x)\dd x
  \]
  introduces a scalar product on its natural domain
  \[
    D\coloneqq \{ f\in L^1_{\mathrm{loc}}(-1,1); \int_{(-1,1)} |f(x)|^2 w(x) \dd x < \infty \}.   
  \] 
  Throughout, let $q\in \N$.
  \begin{deff} 
    Let $w\in W$. A pair $(\omega,r)=((\omega_j)_{j\in\{0,\ldots,q\}},(r_j)_{j\in\{0,\ldots,q\}})
    \in \R^{q+1}\times [-1,1]^{q+1}$ is called \emph{(right-sided) $w$-Gau\ss--Radau 
      quadrature (of order $q$)}, if $-1\leq r_0\leq r_1 \leq \cdots \leq r_q=1$ and 
    for all $p\in \mathcal{P}_{2q}(-1,1)$ we have
    \[
    \int_{-1}^1p(x)w(x)\dd x = \sum_{j=0}^q \omega_j p(r_j).
    \]
  \end{deff}
  \begin{theorem}\label{p:uni} 
    Let $w\in W$. Then there exists a unique (right-sided) $w$-Gau\ss--Radau 
    quadrature $(\omega,r)$ of order $q$. The weights $\omega$ satisfy 
    $0<\omega_j\leq \int_{(-1,1)} w(x) \, \dd x$ for $j\in \{0,\ldots,q\}$ and $r_0>-1$.
  \end{theorem}
  The result being well-known we only refer to \cite{L04} for the corresponding result for Gau\ss-quadrature rule. Note that the methods in \cite{L04} can straightforwardly be adopted to apply to the Gau\ss--Radau quadrature, if one considers the inner product $\langle \cdot,\cdot\rangle_{\tilde w}$ instead of $\langle \cdot,\cdot\rangle_w$, where for $w\in W$ we put $\tilde w \colon x\mapsto (1-x)w(x)$ 
  and observe $\tilde w \in W$.

  Next, we address the continuous dependence of the Gau\ss--Radau points on the 
  weighting function.
  \begin{theorem}\label{t:cd} The mapping
    \begin{equation}\label{e:cd}
      (W,\|\cdot\|_{L^1(-1,1)}) \to \R^{q+1}\times (-1,1]^{q+1}, w\mapsto (\omega(w),r(w))
    \end{equation}
    is continuous, where $(\omega(w),r(w))$ denotes the $w$-Gau\ss--Radau quadrature.   
  \end{theorem}
  \begin{proof}
    Let $(w_n)_{n\in \N}$ be a sequence in $W$ and $w\in W$ such that 
    $w_n\to w$ in $L^1(-1,1)$. By definition and by Theorem \ref{p:uni}, 
    \[
    (\omega(w_n),r(w_n))\in [0,\sup_{k\in \N} 
    \|w_k\|_{L^1(-1,1)}]^{q+1}\times [-1,1]^{q+1}   
    \]
    for all $n\in \N$. Thus, there exists a convergent subsequence for which we 
    re-use the name with limit $(\overline{\omega},\overline{r})$. Note that 
    $\overline{r}_q=1$. Next, let $p\in \mathcal{P}_{2q}(-1,1)$. Then, for $n\in 
    \N$, we obtain
    \begin{align*}
      \int_{-1}^1 p w &= \lim_{n\to \infty} \int_{-1}^1 p w_n\\ 
      & =\lim_{n\to \infty} \sum_{j=0}^q \omega(w_n)_j p(r(w_n)_j)\\ 
      &= \sum_{j=0}^q \overline{\omega}_j p(\overline{r}_j).
    \end{align*}
    Hence, by Theorem \ref{p:uni}, we infer 
    $(\overline{\omega},\overline{r})=(\omega(w),r(w))$, which eventually implies 
    the assertion.
  \end{proof}
  
  \begin{corollary}\label{c:chi1} 
    For $\tau\in \R$ denote $w_\tau\colon x\mapsto \exp(-\rho \tau(x+1))(\in W)$ 
    and let $(\omega^{(\tau)},r^{(\tau)})$ be the $w_\tau$-Gau\ss--Radau quadrature. 
    For $\tau\in\R$, let $\chi_\tau\in \mathcal{P}_{q+1}(-1,1)$ such that
    \[
    \chi_\tau(r^{(\tau)}_j)=0\quad(j\in\{0,\ldots,q\}),\, \chi_\tau(-1)=1.
    \]
    Then, for every compact set $K\subset \R$, we have
    \[
    \sup_{\tau\in K} \int_{-1}^1 \chi_\tau^2w_\tau <\infty.
    \]  
  \end{corollary}
  \begin{proof}
    Assume by contradiction that there exists $(\tau_n)_n$ convergent to some 
    $\tau$ with the property
    \[
    \int_{-1}^1 \chi_{\tau_n}^2w_{\tau_n}\to \infty.
    \]
    Using Theorem \ref{t:cd}, we compute for $n\in \N$
    \begin{align*}
      \int_{-1}^1 \chi_{\tau_n}^2w_{\tau_n} 
      & = \int_{-1}^1 \prod_{j=0}^q 
      \Big(\frac{x-r^{(\tau_n)}_j}{-1-r^{(\tau_n)}_j}\Big)^2 w_{\tau_n}(x)\dd x\\ 
      & \leq \int_{-1}^1 \prod_{j=0}^{q-1} 
      \Big(\frac{x-r^{(\tau_n)}_j}{-1-r^{(\tau_n)}_j}\Big)^2 w_{\tau_n}(x)\dd x\\ 
      & = \sum_{\ell=0}^q \omega^{(\tau_n)}_\ell \prod_{j=0}^{q-1} 
      \Big(\frac{r^{(\tau_n)}_\ell-r^{(\tau_n)}_j}{1+r^{(\tau_n)}_j}\Big)^2\\ 
      & = \omega^{(\tau_n)}_q \prod_{j=0}^{q-1} 
      \Big(\frac{1-r^{(\tau_n)}_j}{1+r^{(\tau_n)}_j}\Big)^2\\ 
      & \to \omega^{(\tau)}_q \prod_{j=0}^{q-1} 
      \Big(\frac{1-r^{(\tau)}_j}{1+r^{(\tau)}_j}\Big)^2\quad(n\to\infty).
    \end{align*}
    But,
    \[
    0\leq \omega^{(\tau)}_q \prod_{j=0}^{q-1} 
    \Big(\frac{1-r^{(\tau)}_j}{1+r^{(\tau)}_j}\Big)^2
    =\int_{-1}^1 \prod_{j=0}^{q-1} 
    \Big(\frac{x-r^{(\tau)}_j}{-1-r^{(\tau)}_j}\Big)^2 w_{\tau}(x)\dd x
    <\infty,
    \]
    which contradicts the assumption.
  \end{proof}
  Let us introduce for a bounded interval $I\subseteq \R$ the mapping
  \begin{align*}
    \phi_I \colon (-1,1) &\to I,\\                   
    x&\mapsto \frac{a+b}{2}+\frac{b-a}{2}x,
  \end{align*}
  where $a\coloneqq \inf I$, $b\coloneqq \sup I$. Further, we set $|I|\coloneqq b-a$. 
  \begin{corollary}\label{c:chi2} 
    For $\tau\in \R$ let $\chi_\tau$ be as in Corollary \ref{c:chi1}. 
    Let $K\geq 0$. Then
    \[
    \sup_{I\subseteq \R\text{ interval}, |I|\leq K} 
    \frac{1}{|I|}\int_I \big(\chi_{|I|}(\phi_I^{-1}(t))\big)^2e^{-2\rho(t-\inf I)}\dd t
    <\infty.
    \]   
  \end{corollary}
  \begin{proof}
    For $I=(a,b)\subseteq \R$ we compute
    \begin{align*}
      &\frac{1}{b-a}\int_{a}^b \big(\chi_{|I|}(\phi_I^{-1}(t))\big)^2 e^{-2\rho(t-a)} \dd t \\
      &=\frac{1}{b-a}\int_{-1}^1 \big(\chi_{|I|}(x)\big)^2 e^{-2\rho(\phi_I(t)-a)}\phi_I'(x) \dd x\\ 
      &=\frac{1}{2}\int_{-1}^1 \big(\chi_{|I|}(x)\big)^2 e^{-\rho((b-a)(x+1))} \dd x.
    \end{align*}
    Hence, the assertion follows from Corollary \ref{c:chi1}.
  \end{proof}
  The next corollary is concerned with the lowest Gau\ss--Radau point for different weights:
  \begin{corollary}\label{c:t0} 
    For $\tau\in \R$ let $w_\tau$ and $(\omega^{(\tau)},r^{(\tau)})$ be given as in Corollary 
    \ref{c:chi1}. Let $T>0$. Then there exists $c>0$ such that for all intervals $I\subseteq \R$ 
    with $|I|\leq T$ and $0\leq\tau \leq T$ we have
    \[
    \phi_I(r^{(\tau)}_0)-\inf I \geq c |I|.
    \]   
  \end{corollary}
  \begin{proof}
    We observe that \[\R\ni \tau \mapsto \big(t\mapsto e^{-\rho \tau(t+1)}\big)
    \in (W,\|\cdot\|_{L^1(-1,1)})\] is continuous. Hence, the set 
    \[
    \{ \big(t\mapsto e^{-\rho \tau(t+1)}\big); \tau\in [0,T]\}\subseteq (W,\|\cdot\|_{L^1(-1,1)})
    \]
    is compact. Thus, by the continuous dependence of the Gau\ss--Radau point on 
    the weighting function (see \eqref{e:cd}), we obtain that 
    \[
    \{ (\omega^{(\tau)},r^{(\tau)}); \tau\in [0,T]\}\subseteq \R^{q+1}\times (-1,1]^{q+1}
    \]
    is compact, as well. In particular, there exists $c>0$ with the property 
    $r^{(\tau)}_0-(-1)\geq c$. Hence, we obtain for all $\tau\in[0,T]$ and intervals 
    $I\subseteq \R$ with $|I|\leq T$
    \[
    \phi_I(r^{(\tau)}_0)-\inf I 
    = \phi_I(r^{(\tau)}_0)-\phi_I(-1) 
    = |I|(r^{(\tau)}_0-(-1)) 
    \geq c |I|.\qedhere
    \]
  \end{proof}
  
  \bibliographystyle{plain}

\end{document}